\documentclass[12pt]{amsart}

\usepackage{amsmath, amsfonts, amssymb, amsthm}

\theoremstyle{plain}
\newtheorem{theorem}{Theorem}[section]

\newtheorem{proposition}[theorem]{Proposition}

\newtheorem{definition}[theorem]{Definition}
\newtheorem{condition}[theorem]{Condition}

\theoremstyle{definition}
\newtheorem{remark}[theorem]{Remark}

\newcommand{\R}{\mathbb{R}}
\newcommand{\Z}{\mathbb{Z}}
\newcommand{\N}{\mathbb{N}}
\renewcommand{\P}{\mathbf{P}}

\newcommand{\Laplace}{\Delta}

\renewcommand{\div}{\operatorname{div}}

\newcommand{\abs}[1]{\left|#1\right|}
\newcommand{\norm}[1]{\left|\left|#1\right|\right|}

\newcommand*{\bigtimes}{\mathop{\raisebox{-.5ex}{\hbox{\huge{$\times$}}}}}

\newcounter{mycount}

\newenvironment{romlist}{\begin{list}{\it \roman{mycount})}%
   {\usecounter{mycount}\labelwidth=1cm\itemsep 0pt}}{\end{list}}

\author[Dirr]{N.\ Dirr}
\address[N.\ Dirr]{Department of Mathematical Sciences, University of Bath, Bath BA2 7AY, UK}
\email{n.dirr@maths.bath.ac.uk}
\urladdr{http://www.maths.bath.ac.uk/$\sim$nd235/}

\author[Dondl]{P.\ W.\ Dondl}
\address[P.\ W.\ Dondl]{Institute for Applied Mathematics, Universit\"at Bonn, Endenicher Allee 60, D-53115 Bonn, Germany}
\email{pwd@hcm.uni-bonn.de}
\urladdr{http://www.dondl.net/}

\author[Scheutzow]{M.\ Scheutzow}
\address[M.\ Scheutzow]{Fakult\"at II, Institut f\"ur Mathematik, Sekr.\ MA 7--5,
Technische Universit\"at Berlin, Strasse des 17.\ Juni 136,
D-10623 Berlin, Germany} \email{ms@math.tu-berlin.de}
\urladdr{http://www.math.tu-berlin.de/$\sim$scheutzow/}

\keywords{QEW, phase boundaries, pinning, random environment}
\subjclass[2000]{35R60, 74N20}

\title{Pinning of interfaces in random media}
\date{\today}

\begin{document} \begin{abstract}
For a model for the propagation of a curvature sensitive interface in a time independent random medium, as well as for a linearized version which is commonly referred to as Quenched Edwards-Wilkinson equation, we prove existence of a stationary positive supersolution at non-vanishing applied load. This leads to the emergence of a hysteresis that does not vanish for slow loading, even though the local evolution law is viscous (in particular, the velocity of the interface in the model is linear in the driving force).
\end{abstract}
 \maketitle
\section{Introduction}
Problems of interface evolution in heterogeneous media arise in a large number of physical models. Common to such models is a regularizing operator, for example line tension, and the competition between an external applied driving force $F$ and a force field $f(x,y)$ describing the inhomogeneities. Assuming a viscous law for the relation between the driving force and the velocity of the interface, an important question is whether rate independent hysteresis can emerge in an average sense from the interaction between the heterogeneous force field and the regularizing operator.

In this article, we consider a model for the evolution of an interface driven by its mean curvature through a random field of obstacles. Let $n \in \N, n \ge1$. Let $(\Omega, \mathcal{F}, \P)$ be a probability space, $\omega \in \Omega$. We model the interface as the graph $(x, u(x,t,\omega))$ of
a function $u \colon \R^n\times\R \times\Omega \to \R$ moving through a field $f(x,y,\omega)$ of (soft) random obstacles and a constant driving force $F$. More precisely, we consider the PDE
\begin{align}
\partial_t u(x,t,\omega) &=\sqrt{1+\abs{\nabla u(x,t,\omega)}^2} \div \left( \frac{\nabla u(x,t,\omega)}{n\sqrt{1+\abs{\nabla u(x,t,\omega)}^2}}  \right) \label{eq:MCF} \\ 
&+ \sqrt{1+\abs{\nabla u(x,t,\omega)}^2} \left(f(x,u(x,t,\omega)) + F \right) \nonumber \\
&=: \sqrt{1+\abs{\nabla u(x,t,\omega)}^2} \left( \kappa(u(x,t,\omega)) + f(x,u(x,t,\omega)) + F \right), \nonumber\\
u(x,0,\omega) &= 0. \label{eq:icMCF}
\end{align}
The first term on the right hand side of equation~\eqref{eq:MCF} is the mean curvature operator for a surface that is given as the graph of the function $u$. The second term is the driving force, split up into the non-homogeneous random part $f$ and the external constant loading $F$. The random field $f$ will be specified in Section~\ref{sec:construction} in Condition~\ref{cond:obst} and~\ref{hyp:random_obstacle_distribution}. Basically we will assume that the nonhomogeneity consists of randomly distributed individual obstacles of a fixed smooth shape and possibly random strength. By $\kappa(u(x,t,\omega))$ we denote the mean curvature operator for the graph of a function $u(\cdot,t,\omega)$ evaluated at $x$.

Equation~\eqref{eq:MCF} is motivated in the following way (see also~\cite{Coville_10a}): A very basic model for an interface (phase boundary, 
dislocation line in its slip plane etc) moving through an array of random obstacles (e.g. impurities, other
dislocation lines) in an over-damped limit (inertial effects are neglected) is the gradient flow
of the area functional plus a random bulk term. Consider thus a bounded set $U \subset \R^{n+1}$ and a smooth hypersurface $\Gamma$ that is the boundary of the set $A_\Gamma \subset U$ and define the energy
\begin{equation}
\label{eq:energy}
E(\Gamma) := \mathcal{H}^1(\Gamma) + \int_{A_\Gamma}( f(X,\omega) + F) \,\mathrm{d}X.
\end{equation}
Here, $\mathcal{H}^1$ denotes the 1-dimensional Hausdorff measure. The first inner variation (i.e., deforming the interface with the flow of
a smooth vector field) yields the mean curvature $\kappa$ and a bulk term. The viscous gradient flow with respect to the energy~\eqref{eq:energy} is thus given by the evolution equation 
\begin{equation}
\label{eq:MCF_normalvel}
v_n(x)  = \kappa(x) + f(x,\omega) + F, \quad x\in \Gamma,
\end{equation}
for the normal velocity $v_n$ of the interface. Now it is also possible to extend the notion of an evolving interface to $U = \R^{n+1}$.
The model is called {\em quenched}, because the random field does not explicitly depend on time. For forced mean curvature flow and applications, in particular in the case of periodic forcing, 
we refer to \cite{Cardaliaguet_09a,Craciun_04a,Dirr_08a}. Since we are only interested in constructing a supersolution for~\eqref{eq:MCF_normalvel} when the initial surface is flat, it is sufficient to consider the mean curvature flow for an interface that is the graph of a function $u$, i.e, $\Gamma = \{(x,y): y=u(x), x\in \R^n \}$.

If the gradient of $u$ is sufficiently small, the evolution by forced mean curvature flow (MCF)
for the graph can be approximated heuristically by a semi-linear parabolic PDE of the form
\begin{align}
\partial_t u(x,t,\omega) &= \Laplace u (x,t,\omega) + f(x,u(x,t,\omega),\omega) +F 
\quad \textrm{on $\R^n$}, \label{eq:QEW} \\
u(x,0,\omega) &= 0. \label{eq:icQEW}
\end{align}
These kinds of problems have found considerable interest in the physics community, see e.g. \cite{Kardar_97a, Kleemann_04a, Brazovskii_04a}.
They are often referred to as the Quenched Edwards-Wilkinson (QEW) model.

The goal of this article is to construct, for some $F>0$, a stationary supersolution $v$ to~\eqref{eq:MCF} and to \eqref{eq:QEW} satisfying $v\ge0$. In this article, we consider the natural case where, due to the randomness of the obstacle field, there exist arbitrarily large simply connected 
domains with positive combined driving force (i.e., where $f+F>0$), in other words, large areas without obstacle. This makes a purely analytical approach, as employed in~\cite{Dirr_06a, Dirr_08a}, insufficient.

In order to illustrate the difficulty, consider in 1+1 dimensions
a periodic array of obstacles with a forcing $F>0.$ 
Now remove each obstacle independently with rate $p\ll 1$ and increase the obstacle strength
by $\delta(p)$ in order to keep the expected value of the obstacle strength equal to the periodic one.
The resulting random configuration (periodic with random ``holes'') 
may have a solution which is unbounded as $t\to\infty$: For any $h\in \Z$ and
$l\in \N,$ there exists almost surely  a $z\in \Z$ such that the ball $B((z,h),l)$ of radius $l$ centered at $(z,l)$ is free
of obstacles. For sufficiently large $l,$  the solution in this ball will grow to a height that is larger than $h+1.$
For an appropriate choice of obstacles and parameters, this perturbation can grow like a kink-antikink pair
in a reaction diffusion equation (e.g. Allen-Cahn) with forcing, until the entire curve has moved up at least 
one unit. Then the process repeats itself at a ``hole'' at height level $h+1.$

If such a supersolution exists, by the comparison principle for the mean curvature flow and  for parabolic equations, an evolving solution $u$ with any initial condition below $v$ will always remain below $v$---the interface is pinned. Such pinning of an interface leads to a hysteresis that does not vanish for slow loading in the physical system. To see this, consider a loading cycle starting with $F=0$, increasing at first. The interface remains pinned until the driving force reaches a critical value (see Section~\ref{sec:conclusion} for a brief discussion of depinning). Above the critical force the material transforms (switching polarization, for example). Upon reversal of the driving force, the
same phenomenon occurs\footnote{There is a difficulty in modeling this behavior. In a physical situation, the obstacles (non-transforming inclusions, for example) will always obstruct the evolution of an interface moving in any direction. A reasonable way to express this would be to consider the
driving force $f(x,u(x,t,\omega),\omega)\cdot\operatorname{sign}(\partial_t u)$. Such an additional nonlinearity in the equation would make the analysis unnecessarily complicated---we thus restrict ourselves to treating the transformation ($\partial_t u \ge 0$) and the back-transformation ($\partial_t u \le 0$) separately.}. One can see that the plotting the transformed region vs.~the driving force will show a hysteresis loop that does
not vanish even for slowly varying driving force $F$.

In the following section, we construct such a supersolution under suitable assumptions on the random obstacle field. Section~\ref{sec:conclusion} provides some outlook.

\section{Construction of a supersolution}
\label{sec:construction}
We first pose two conditions that will fix the structure of the non-homogeneous term $f$ in equations~\eqref{eq:MCF} and~\eqref{eq:QEW}.
This random nonlinearity $f$ is constructed in the following way: We consider an obstacle function $\phi \in C^\infty(\R^n\times\R)$ with the properties
\begin{condition}[Obstacle shape]
\label{cond:obst}
There exist $r_1, r_0$ with $r_1>\sqrt{n} r_0>0$, so that
\begin{romlist}
\item $\phi \le 0$, $\phi(x,y) = 0$ for $\norm{(x,y)} > r_1$,
\item $\phi(x,y) \le -1$ for $\norm{(x,y)}_\infty \le r_0$.
\end{romlist}
\end{condition}
This fixes a `shape' for the individual obstacles. Here, $\norm{\cdot}$ denotes the Euclidean norm on $\R^{n+1}$, $\norm{\cdot}_\infty$ denotes the maximum-norm. The heterogeneity $f$ is now given as the sum over
individual obstacles with centers $\{ (x_i(\omega), y_i(\omega)) \}_{i\in\N}$, and strength $f_i(\omega) \ge 0$, i.e.,
$$
f(x,y,\omega) = \sum_i f_i(\omega) \phi( x-x_i(\omega), y-y_i(\omega) ).
$$

We now pose a condition on the distribution of obstacles. The condition basically states that there is a uniform lower bound for finding an obstacle of some strength (also bounded from below) in a box of volume 1, independent of its shape or position, and independently for pairwise disjoint boxes.
\begin{condition}[Obstacle distribution]
\label{hyp:random_obstacle_distribution}
The random distribution of obstacle sites $\{(x_k,y_k)\}_{k\in\N} \subset \R^{n}\times [r_1,\infty)$ and strength $\{f_k\}_{k\in\N} \subset [0,\infty)$ satisfy
\begin{romlist}
\item $(x_k,y_k)$ are distributed according to an $n+1$-dimensional Poisson process on $\R^{n}\times [r_1,\infty)$
with intensity $\lambda>0$.
\item $f_k$ are iid strictly positive random variables which are independent of $\{x_k,y_k \}$.
\end{romlist}
\end{condition}
Note that there are no obstacles crossing the line $\{y=0\}$, so at $t=0$ the interface with initial condition~\eqref{eq:icMCF} only sees the external driving force. For a small time, the velocity of the interface is thus uniformly positive for $F>0$. The comparison principle ensures that we thus have $\partial_t u \ge 0$ for all times. To see this, assume that for a smooth solution to~\eqref{eq:QEW} there exists a first time $t_0>0$ when there exists $x_0 \in \R^n: \partial_t u(x_0,t_0) = 0$. Differentiating~\eqref{eq:QEW} with respect to time yields
$$
\partial_{tt} u(x_0,t_0) = \Laplace u_t(x_0,t_0) + f_u(x_0, u(x_0,t_0)) \partial_t u(x_0,t_0) =  \Laplace u_t(x_0,t_0)  \ge 0.
$$
Noting that, for a smooth solution of~\eqref{eq:MCF}, the spatial gradient of $\partial_t u$ also vanishes at $(x_0,t_0)$, one can obtain the same
non-negativity result for mean curvature flow.

As mentioned in the introduction, the main difficulty in this work stems from the fact that there exist, with positive probability, arbitrarily large areas with arbitrarily small obstacles. We thus rescale space into boxes calling them ``open sites'' if they contain a suitably large obstacle. Theorems~\ref{thm:pinning} and~\ref{thm:pinningMCF} below then depend crucially on the existence of an infinite cluster of open sites in $\Z^{n+1}$ that is the graph of a Lipschitz function, for site percolation with $\P\{ \textrm{Site open} \} >p_c$ independently for all sites. It will be clear that the Lipschitz condition is essential for the construction of a supersolution. This result is proved in~\cite{Dondl_09f}, we repeat it here for the reader's convenience. See also~\cite{Grimmett_10b} for an improved estimate on the critical percolation threshold.
\begin{theorem}[Dirr-Dondl-Grimmett-Holroyd-Scheutzow]
\label{thm:LipPerc}
Let $n \ge 1$ and $p\in (0,1)$. We designate $z\in\Z^{n+1}$ \textbf{open} with probability
$p$, and otherwise \textbf{closed}, with different sites receiving
independent states. The corresponding probability measure on
the sample space $\Omega = \{0,1\}^{\Z^{n+1}}$ is denoted by $\P_p$.
We write $\|\cdot\|_1$ for the $1$-norm on $\Z^{n+1}$. The following holds:

For any $n\geq 1$, if $p> 1-(2n+2)^{-2} =: p_c$ then there exists
a.s.\ a (random) function $L:\Z^{n}\to \N$ with the
following properties.
\begin{romlist}
\item For each $x\in \Z^{n}$, the site $(x,L(x))\in\Z^{n+1}$
    is open.
\item For any $x,y\in \Z^{n}$ with $\|x-y\|_1=1$ we have
    $|L(x)-L(y)|\leq 1$.
\item For any isometry $\theta$ of $\Z^{n}$ the functions
    $L$ and $L\circ \theta$ have the same laws, and the
    random field $(L(x): x\in\Z^{n})$ is ergodic under each translation of
    $\Z^{n}$.
\item There exists $A=A(p,d) < \infty$ such that
$$
\P_p(L(0)>k)\leq A \nu^k, \quad k\ge 0.
$$
where $\nu=(2n+2)(1-p)<1$.
\end{romlist}
\end{theorem}

We first show the result asserting the existence of a stationary positive supersolution to the semilinear equation~\eqref{eq:QEW}, since the calculations are somewhat simpler. 
\begin{theorem}[Existence of a pinned solution for QEW]
\label{thm:pinning}
If Conditions~\ref{cond:obst} and~\ref{hyp:random_obstacle_distribution} are satisfied, then there exists $F^*>0$ and a non-negative $v \colon \R^n\times\Omega \to [0,\infty)$ so that $0 \ge \Laplace v(x,\omega) +  f(x, v(x,\omega), \omega) + F^*$ a.s., i.e., any solution to~\eqref{eq:QEW} with initial condition~\eqref{eq:icQEW} and $F\le F^*$ gets pinned.
\end{theorem}
The proof consists of a piecewise construction of the supersolution, so it is first necessary to give some estimates on the components that will be used. We denote by $B_r$ the open unit ball of radius $r$ around $0$.

\begin{definition}[Local solution]
\label{def:local}
Given $r_\mathrm{out} > r_\mathrm{in} > 0$, $F_\mathrm{in}>0$, and $F_\mathrm{out}<0$, let $v_\mathrm{in}$ be the unique solution of $\Laplace v_\mathrm{in} = F_\mathrm{in}$ on $B_{r_\mathrm{in}} \subset \R^{n}$, $v_\mathrm{in} = 0$ on $\partial B_{r_\mathrm{in}}$. Let $v_\mathrm{out}$ be the unique solution of $\Laplace v_\mathrm{out} = F_\mathrm{out}$ on $B_{r_\mathrm{out}} \setminus B_{r_\mathrm{in}} \subset \R^{n}$, with boundary conditions $v_\mathrm{out} = 0$ on $\partial B_{r_\mathrm{in}}$ and
$\nu \cdot \nabla v_\mathrm{out} = 0$ on $\partial B_{r_\mathrm{out}}$.

We define
\begin{equation*}
v_\mathrm{local} := \left\{ \begin{array}{ll}
v_\mathrm{in} & \textrm{on $B_{r_\mathrm{in}}$}, \\
v_\mathrm{out} & \textrm{on $B_{r_\mathrm{out}} \setminus B_{r_\mathrm{in}}$}, \\
\lim_{r\to r_\mathrm{out}} v_\mathrm{out}(r) & \textrm{on $\partial B_{r_\mathrm{out}}$}, \\
\infty & \textrm{otherwise}. \end{array}\right.
\end{equation*}
\end{definition}

\begin{proposition}
\label{prop:local_sol}
The function $v_\mathrm{local}$ defined above satisfies
\begin{romlist}
\item $v_\mathrm{local}$ is radially strictly increasing on $B_{r_\mathrm{out}}$,
\item the graph of $v_\mathrm{local}$ restricted to $B_{r_\mathrm{in}}$ is contained in the set $B_{r_\mathrm{in}} \times [- \frac{F_\mathrm{in}}{2n}r_\mathrm{in}^2, 0]$
\item given $\bar{f}>0$, if
\begin{equation}
\label{eq:jump_cond}
F_\mathrm{in} r_\mathrm{in} \ge \abs{F_\mathrm{out}} \left(-r_\mathrm{in}+ \frac{r_\mathrm{out}^n}{r_\mathrm{in}^{n-1}} \right),
\end{equation}
and if $\max\{r_\mathrm{in},\frac{F_\mathrm{in}}{2n}r_\mathrm{in}^2\} \le r_0$, then $v_\mathrm{local}$ satisfies, in the sense of distributions (and in the sense of viscosity solutions),
\begin{equation}
\label{eq:local_supers}
0 \ge \Laplace v_\mathrm{local} + \bar{f} \phi(\cdot, v_\mathrm{local}(\cdot)+r_0) +F \quad \textrm{on $B_{r_\mathrm{out}}$}
\end{equation}
for all $F \le \min\{-F_\mathrm{out}, \bar{f}-F_\mathrm{in}\}$.
\end{romlist}
\end{proposition}
\begin{proof}
The individual assertions are proved by a simple calculation.
\begin{romlist}
\item Follows immediately from the maximum principle.
\item The function $v_\mathrm{in}$ is nothing but a parabola, namely $v_\mathrm{in}(x) = \frac{F_\mathrm{in}}{2n}\abs{x}^2 - \frac{F_\mathrm{in}}{2n}r_\mathrm{in}^2$. The assertion can be read off this form.
\item From {\it ii)} and from the assumption on $\phi$ in~\ref{cond:obst}, one can see that for $x\in B_{r_\mathrm{in}}$, we have $\phi(x, v_\mathrm{local}(x)+r_0) \le -1$. The property~\eqref{eq:local_supers} for each individual piece of $v_\mathrm{local}$ can then be seen directly from the definition of $v_\mathrm{local}$. The assertion follows by noting that equation~\eqref{eq:jump_cond} implies that the first derivative jumps down going radially outward across $\partial B_{r_\mathrm{in}}$\footnote{The term $\frac{F_\mathrm{in}}{n} r_\mathrm{in}$ is the radial derivative of $v_\mathrm{in}$ and 
$-F_\mathrm{out} \left(\frac{-r_\mathrm{in}}{n} + \frac{r_\mathrm{out}^n}{n r_\mathrm{in}^{n-1}} \right)$ is the radial derivative of $v_\mathrm{out}$ at $\partial B_{r_\mathrm{in}}$ }. This implies that $v_\mathrm{local}$ is the pointwise minimum of two supersolution, thus a supersolution itself. The Laplacian is then a negative measure.
\end{romlist}
\end{proof}

\begin{definition}[Rescaling]
Given $l>2r_1$, $d>0$, $h >0$, and $k = (k_1, k_2, \dots, k_n) \in \Z^n$, $j\in\N$, let 
\begin{romlist}
\item $\tilde{Q}_k := \bigtimes_{i=1}^n [k_i (l+d) +r_1, k_i (l+d)+l-r_1]$,
\item $\overline{Q}_k := \bigtimes_{i=1}^n [k_i (l+d), k_i (l+d)+l]$,
\item $\overline{Q} := \bigcup_{k} \overline{Q}_k$ and $\overline{D} := \R^n \setminus \overline{Q}$,
\item $\tilde{Q}_{k,j} = \tilde{Q}_k \times [(j-1) h+r_1, jh+r_1]$.
\end{romlist}
Here $\bigtimes_{i=1}^n$ denotes the cartesian product of the $n$ intervals following.
\end{definition}
\begin{remark}
The sets $\overline{Q}$ and $\overline{D}$ split $\R^n$ into cubes, each separated by a distance $d$. The reduced cubes $\tilde{Q}_k$ are smaller
by the length $2r_1$ in every dimension, so that an obstacle with center in $\tilde{Q}_k$ fits completely inside $\overline{Q}_k$. The sets $\tilde{Q}_{k,j}$
are extended in the $n+1$-st direction by a height $h$.
\end{remark}

\begin{proposition}[Percolating obstacles]
\label{prop:perc_obst}
Given $h>0$, fix $l(h)>0$ and $\bar{f}>0$ s.t. 
$$
1-\exp\{-\lambda \abs{A} \cdot \P\{f_1\ge\bar{f} \} \} > p_c
$$
for $\abs{A} = (l-2r_1)^n h$, i.e., $l(h) = C_0(p_c, \lambda, \P\{f_1\ge \bar{f} \}) h^{-1/n} + 2r_1$. Then there exists a
random function $L\colon \Z^n \to \N$ with Lipschitz constant 1, such that, a.s., for all $k\in \Z^n$, there exists $i \in \N$ such that
$(x_i, y_i) \in \tilde{Q}_{k,L(k)}$ and $f_i \ge \bar{f}$.

For each $k\in \Z^n$ we select one obstacle index $i\in \N$ with the above property and collect these obstacle indices in the set $I$.
\end{proposition}
\begin{proof}
This is a direct consequence from Theorem~\ref{thm:LipPerc}. Indeed, considering a cuboid $\tilde{Q}_{k,j}$ open if it contains an obstacle
of strength greater or equal $\bar{f}$, we find that a cuboid is open with probability greater than $p_c$.
\end{proof}

\begin{definition}[Flat supersolution]
\label{def:vflat}
We define the flat supersolution $v_\mathrm{flat} \colon \R^n \to \R$ as 
$v_\mathrm{flat}(x) := \min_{i \in I} v_\mathrm{local}(x-x_i)$.
\end{definition}
\begin{proposition}
\label{prop:flat}
Fix $h>0$, $\bar{f}>0$, and $l(h)$ as in Proposition~\ref{prop:perc_obst}. Let $r_\mathrm{out} = \sqrt{n}(l(h)+\frac{d}{2}-r_1)$ and assume that
$r_\mathrm{out}$, $r_\mathrm{in}$, $F_\mathrm{in}$ and $F_\mathrm{out}$ satisfy the conditions in Proposition~\ref{prop:local_sol}. Then 
$v_\mathrm{flat}$ satisfies, a.s., in the sense of distributions (and in the sense of viscosity solutions)
\begin{equation*}
0 \ge \Laplace v_\mathrm{flat}(x) + \sum_{i\in I} \bar{f} \phi(x-x_i, v_\mathrm{flat}(x)+r_0) +F \quad \textrm{on $\R^n$}
\end{equation*}
for all $F \le \min\{-F_\mathrm{out}, \bar{f}-F_\mathrm{in}\}$.
\end{proposition}
\begin{proof}
Since $v_\mathrm{flat}$ is a minimum over shifted copies of the function $v_\mathrm{local}$, which is a supersolution where it is not equal $+\infty$ as proved in
Proposition~\ref{prop:local_sol}, it is enough to show that $v_\mathrm{flat}(x) < +\infty$ for all $x\in \R^n$. This is, however, true by the
choice of $r_\mathrm{out} = \sqrt{n}(l(h)+\frac{d}{2}-r_1)$ with the property that the union over all $k\in \Z^n$ of closed balls of this radius with centers anywhere in $\tilde{Q}_k$ cover all of $\R^n$.
\end{proof}
\begin{remark}
Since the function $v_\mathrm{local}$ is strictly increasing on $B_{r_\mathrm{out}}$, the minimization process assigns each obstacle center $x_i$
its Voronoi cell. On the Voronoi cell associated with $x_i$, the function $v_\mathrm{local}$ centered at $x_i$ attains the minimum.
\end{remark}

\begin{proposition}[Gluing function]
\label{prop:glue}
Fix $h>0$, $d>0$, $l>0$. Let $L\colon \Z^n \to \R$ be a function with the property that if $x,y\in \Z^{n}$ with $\|x-y\|_1=1$ we have
$|L(x)-L(y)|\leq 2h$. Then there exists $C_1>0$, depending only on the dimension $n$, such that there exists a smooth function
$v_\mathrm{glue}\colon \R^n\to \R$ such that
\begin{romlist}
\item for all $k\in \Z^n$, $v_\mathrm{glue}(x) = L(k)$ if $x\in \overline{Q}_k$,
\item $\operatorname{supp} \nabla v_\mathrm{glue} \subset \overline{D}$,
\item $\norm{D^2  v_\mathrm{glue}}_\infty \le C_1\frac{h}{d^2}$,
\item $\norm{\nabla  v_\mathrm{glue}}_\infty \le C_1\frac{h}{d}$.
\end{romlist}
\end{proposition}
\begin{proof}
It suffices to take a piecewise constant function that changes values on the center hyperplanes of the set $\overline{D}$ and apply a standard mollifier of
size $d/2$.
\end{proof}

We now collected all the components to construct the supersolution.
\begin{proof}[Proof of Theorem~\ref{thm:pinning}]
First, fix $\bar{f}$ and the function $l(h)$ as in Proposition~\ref{prop:perc_obst}. Then fix $0<F_\mathrm{in}<\frac{\bar{f}}{2}$ 
and $r_\mathrm{in}$ such that $\max\{r_\mathrm{in},\frac{F_\mathrm{in}}{2n}r_\mathrm{in}^2\} \le r_0$.
 
We now need to find $d>0$ and $h>0$, such that, with $r_\mathrm{out}$ chosen as in Proposition~\ref{prop:flat}, 
\begin{align}
F_\mathrm{in} r_\mathrm{in} \ge\;& \abs{F_\mathrm{out}} \left(-r_\mathrm{in} + \frac{r_\mathrm{out}^n}{ r_\mathrm{in}^{n-1}} \right)   \label{eq:v_in_v_out_discontinuity} \\
=\;&\abs{F_\mathrm{out}} \left(-r_\mathrm{in} + \frac{1}{ r_\mathrm{in}^{n-1}} \left(\sqrt{n}\left(l(h)+\frac{d}{2} - r_1\right)\right)^{n} \right)  \nonumber \\
=\;& \abs{F_\mathrm{out}} \left(-r_\mathrm{in} + \frac{1}{ r_\mathrm{in}^{n-1}}  \left(\sqrt{n}\left(C_0 h^{-1/n}+ r_1+\frac{d}{2}\right)\right)^{n} \right)\nonumber
\end{align}
and, at the same time,
\begin{equation}
\label{eq:gluing_bound}
\abs{F_\mathrm{out}} \ge 2 C_1\frac{h}{d^2} \ge 2\norm{\Laplace v_\mathrm{glue}}_\infty.
\end{equation}

Putting the two together, and using the fact that $F_\mathrm{in}$ and $r_\mathrm{in}$ are now fixed, it is sufficient to choose $d$ and $h$ so that
\begin{equation}
\label{eq:connection}
C' > C \frac{h}{d^2} \left(h^{-1/n}+\frac{d}{2} + r_1 \right)^{n}.
\end{equation}
Since one has 
$C \frac{h}{d^2} \left(h^{-1/n}+\frac{d}{2} +r_1\right)^{n} < 2^n C \left( \frac{1}{d^2}  + \frac{h}{d^2}\left(\frac{d}{2} + r_1\right)^{n} \right)$,
one can see that there is such a choice. Now we fix $F_\mathrm{out} = -2 C_1\frac{h}{d^2}$.

Now we choose, according to Proposition~\ref{prop:perc_obst}, the index set $I$ of relevant obstacles.
From Proposition~\ref{prop:glue} and the Lipschitz condition on the percolating cluster of selected boxes from Proposition~\ref{prop:perc_obst}, we know there exists a function $v_\mathrm{glue}$ whose derivative is only supported on the set $\tilde{D}$ and the property 
$v_\mathrm{glue}(x_k) = y_k+r_0$ for all $k \in I$, and $\norm{\Laplace v_\mathrm{glue}}_\infty \le C_1 \frac{h}{d^2}$.

Choosing $0< F^* \le \min\{-\frac{F_\mathrm{out}}{2}, \frac{\bar{f}}{2}\} $ One can now see that the function 
\begin{equation} 
v = v_\mathrm{flat} + v_\mathrm{glue}
\end{equation}
satisfies 
\begin{equation}
0 \ge \Laplace v(x) +  \sum_{i\in I} \bar{f} \phi(x-x_i, v(x)-y_i) + F^* \ge \Laplace v(x) +  f(x, v(x), \omega) + F^*.
\end{equation}
\end{proof}

\begin{remark}
\label{rem:lattice}
For a slightly different model, if the pinning sites are centered on a regular lattice, i.e., $f(x,y,\omega)= \sum_{i \in \Z^{n}, j\in \Z+1/2} f_{i,j}(\omega) \phi(x-i, y-j)$, there is a lower bound for $h$---one can not make the boxes more shallow than the lattice
spacing. This leads to the fact that there might not exist a $d$ satisfying the estimate~\eqref{eq:connection}. For $n=1$, one can still find
the supersolution in the described way, since the scaling of the gradient of $v_\mathrm{out}$ with the distance $d$ works favorably. In particular, the 
construction works for the model used in~\cite{Coville_10a}.

For $n \ge 2$,
the construction only works for $\bar{f}$ sufficiently large. Depending on distribution of
$f_1$, such a choice for $\bar{f}$ might not be possible.
\end{remark}

We now turn towards the construction of a supersolution for the mean curvature flow. The theorem itself remains unchanged. 
\begin{theorem}[Existence of a pinned solution for MCF]
\label{thm:pinningMCF}
If Conditions~\ref{cond:obst} and~\ref{hyp:random_obstacle_distribution} are satisfied, then there exists $F^*>0$ and a non-negative $w \colon \R^n\times\Omega \to [0,\infty)$ so that  a.s., $0 \ge \kappa(w(x,\omega)) + f(x,w(x,\omega)) + F^*$ a.s.,
i.e., any solution to~\eqref{eq:MCF} with initial condition~\eqref{eq:icMCF} and $F\le F^*$ gets pinned.
\end{theorem}
The idea of the proof is to construct a local solution and a gluing function for the mean curvature operator and then provide estimates akin to Propositions~\ref{prop:local_sol} and~\ref{prop:glue} for these functions. The rest of the proof, modulo an estimate for the behavior of the nonlinear mean curvature operator when adding the local solution and the gluing function, can remain unchanged. 

\begin{definition}[Local solution for MCF]
As in Definition~\ref{def:local}, fix $r_\mathrm{in} \in (0,r_0)$, $r_\mathrm{out}>0$, $F_\mathrm{in}>0$, and $F_\mathrm{out}<0$, but now making sure that $r_\mathrm{in} \le F_\mathrm{in}$ and $\abs{F_\mathrm{out}}$ is sufficiently small so that $\frac{(r_\mathrm{out} - r_\mathrm{in})^{n-1}}{r_\mathrm{out}^n-(r_\mathrm{out}- r_\mathrm{in})^n} >  \abs{F_\mathrm{out}}$. We construct the local solution from rotationally symmetric surfaces of constant mean curvature, so called Delauney-Surfaces~\cite{Delaunay_41a}.

Let $w_\mathrm{in}\colon B_{r_\mathrm{in}} \to \R$ be given as $w_\mathrm{in}(x) = -\sqrt{F_\mathrm{in}^2 -\abs{x}^2}+\sqrt{F_\mathrm{in}^2-r_\mathrm{in}^2}$. The radially symmetric function $w_\mathrm{out} \colon B_{r_\mathrm{out}} \setminus B_{r_\mathrm{in}} \to \R$ is defined by an elliptic integral as
\begin{equation*}
w_\mathrm{out}(r) =  \int_0^{r_\mathrm{out}-r} \frac{-1}{\sqrt{\frac{(r_\mathrm{out} - \rho)^{2n-2}}{(r_\mathrm{out}^n-(r_\mathrm{out}-\rho)^n)^2 F_\mathrm{out}^2}-1}} \,\mathrm{d}\rho - C,
\end{equation*}
where $C =  \int_0^{r_\mathrm{out}-r_\mathrm{in}} \frac{-1}{\sqrt{\frac{(r_\mathrm{out} - \rho)^{2n-2}}{(r_\mathrm{out}^n-(r_\mathrm{out}-\rho)^n)^2 F_\mathrm{out}^2}-1}} \,\mathrm{d}\rho$.

We define
\begin{equation*}
w_\mathrm{local} := \left\{ \begin{array}{ll}
w_\mathrm{in} & \textrm{on $B_{r_\mathrm{in}}$}, \\
w_\mathrm{out} & \textrm{on $B_{r_\mathrm{out}} \setminus B_{r_\mathrm{in}}$}, \\
\lim_{r\to r_\mathrm{out}} w_\mathrm{out}(r) & \textrm{on $\partial B_{r_\mathrm{out}}$}, \\
\infty & \textrm{otherwise}. \end{array}\right.
\end{equation*}
\end{definition}

\begin{proposition}
We have $w_\mathrm{local}$ is finite on $B_{r_\mathrm{out}}$. Furthermore, it holds that
\begin{equation}
0 \ge  \kappa(w_\mathrm{local}) + \bar{f} \phi(\cdot, w_\mathrm{local}(\cdot)+r_0) +F \quad \textrm{on $B_{r_\mathrm{out}}$}
\end{equation}
in the sense of viscosity solutions, if
\begin{romlist}
\item $0 \ge F_\mathrm{in}-\bar{f}+F$,
\item $0 \ge F_\mathrm{out}  + F$,
\item $F_\mathrm{in} - \sqrt{F_\mathrm{in}^2 -r_\mathrm{in}^2} \le r_0$.
\item $\frac{1}{\sqrt{\frac{(r_\mathrm{out} - r_\mathrm{in})^{2n-2}}{(r_\mathrm{out}^n-(r_\mathrm{out}-r_\mathrm{in})^n)^2 F_\mathrm{out}^2}-1}}
< \frac{r_\mathrm{in}}{\sqrt{F_\mathrm{in}^2-r_\mathrm{in}^2}}$.
\end{romlist}
\end{proposition}
\begin{proof}
The first statement is clear by inspection, since under the conditions on $r_\mathrm{in}$, $F_\mathrm{in}$, $r_\mathrm{out}$, and $F_\mathrm{out}$ the functions $w_\mathrm{in}$ and $w_\mathrm{out}$ remain finite. The second statement holds due to {\it i)} and {\it ii)} on the inside of the sphere $B_{r_\mathrm{in}}$ (Condition {\it iii)} confines the graph of $w_\mathrm{in}$ to the set where $\phi\le -1$) and on the inside of the annulus $B_{r_\mathrm{out}} \setminus B_{r_\mathrm{in}}$. Condition {\it iv)} ensures that the derivative of $w_\mathrm{local}$ jumps only downwards going radially across the boundary from the sphere to the annulus, so that the mean curvature of $w_\mathrm{local}$ at the boundary is negative in the viscosity sense.
\end{proof}
\begin{proof}[Proof of Theorem~\ref{thm:pinningMCF}]
In order to employ the construction from the proof of Theorem~\ref{thm:pinning}, we first  need to make sure that the scaling of 
$$
\abs{\partial_r w_\mathrm{out}(r)|_{r=r_\mathrm{in}}} = \frac{1}{\sqrt{\frac{(r_\mathrm{out} - r_\mathrm{in})^{2n-2}}{(r_\mathrm{out}^n-(r_\mathrm{out}-r_\mathrm{in})^n)^2 F_\mathrm{out}^2}-1}}
$$
 is suitable. Consider thus $F_\mathrm{out} = -c \frac{r_\mathrm{in}^{n-1}}{r_\mathrm{out}^n}$ and note that $\abs{\partial_r w_\mathrm{out}(r)|_{r=r_\mathrm{in}}}$ is then
 given by
\begin{equation*}
g(r_\mathrm{out},c) := \frac{1}{\sqrt{\frac{(r_\mathrm{out} - r_\mathrm{in})^{2n-2}{r_\mathrm{out}^{2n}}}{(r_\mathrm{out}^n-(r_\mathrm{out}-r_\mathrm{in})^n)^2 c^2 r_\mathrm{in}^{2n-2}}-1}}.
\end{equation*}
One can see by a simple calculation that for $c$ small enough there exists $C_2>0$ so that  $g(r_\mathrm{out},c)<C_2 c$. This, however, implies that in the correct regime the Delaunay-Surface from the construction of $w_\mathrm{out}$ admits the same scaling properties with respect to $r_\mathrm{out}$ as the function $v_\mathrm{out}$.


The proof of Theorem~\ref{thm:pinningMCF} can now be completed by the same construction as for the semilinear equation. First, fix the supersolution inside the obstacles. This determines the maximal outgoing radial derivative $\partial_r w_\mathrm{in}(r)|_{r=r_\mathrm{in}} := G$ and the radius of the inner sphere $r_\mathrm{in}$. Consider then the function $w_\mathrm{flat}$ constructed analogously to $v_\mathrm{flat}$ above. It is necessary to satisfy (after setting $r_\mathrm{out} = \sqrt{n}(l(h)+\frac{d}{2}-r_1)$)
$$
g(r_\mathrm{out},c) < G
$$
as well as 
\begin{equation}
\label{eq:MCF_scaling}
\abs{F_\mathrm{out}} = c \frac{r_\mathrm{in}^{n-1}}{\left(\sqrt{n}(C_0 h^{-1/n}+\frac{d}{2}-r_1)\right)^n} \ge 2C_1\frac{h}{d^2}.
\end{equation}
The scaling property discussed above ensures that this is possible.

It remains to show that adding the function $v_\mathrm{glue}$ from Proposition~\ref{prop:glue} does not destroy the property of negative mean curvature. Define
$\nu(u) := \sqrt{1+\abs{\nabla u}^2}$. We have, after collecting terms from expanding the divergence in the mean curvature operator,
\begin{align*}
\kappa(w_\mathrm{out} + v_\mathrm{glue})
=\;&  \frac{\Laplace (w_\mathrm{out} + v_\mathrm{glue}) } {\nu(w_\mathrm{out} + v_\mathrm{glue})} \\
&+
\frac{\left( (D^2 w_\mathrm{out} + D^2 v_\mathrm{glue}) \cdot (\nabla w_\mathrm{out} + \nabla v_\mathrm{glue}), \nabla w_\mathrm{out} + \nabla v_\mathrm{glue}\right)}{\nu(w_\mathrm{out} + v_\mathrm{glue})^3} \\
=\;& \kappa(w_\mathrm{out}) \\
&+ \frac{\Laplace v_\mathrm{glue}}{\nu(w_\mathrm{out} + v_\mathrm{glue})} \\
&+ \Laplace u \left(\frac{1}{\nu(w_\mathrm{out} + v_\mathrm{glue})} - \frac{1}{\nu(w_\mathrm{out})}\right)\\
&+ (D^2 w_\mathrm{out} \cdot \nabla w_\mathrm{out}, \nabla w_\mathrm{out})
 \left(\frac{1}{\nu(w_\mathrm{out} + v_\mathrm{glue})^3} - \frac{1}{\nu(w_\mathrm{out})^3}  \right) \\
&+ \frac{(D^2 v_\mathrm{glue} \cdot (\nabla w_\mathrm{out} + \nabla v_\mathrm{glue}), \nabla w_\mathrm{out} + \nabla v_\mathrm{glue})}{\nu(w_\mathrm{out} + v_\mathrm{glue})^3} \\
&+ \frac{ 2(D^2 w_\mathrm{out} \cdot \nabla w_\mathrm{out},\nabla v_\mathrm{glue})}{\nu(w_\mathrm{out} + v_\mathrm{glue})^3} \\
&+ \frac{(D^2 w_\mathrm{out} \cdot \nabla v_\mathrm{glue}, \nabla v_\mathrm{glue})}{\nu(w_\mathrm{out} + v_\mathrm{glue})^3} 
\end{align*}
\begin{align*}
=\;&\kappa(w_\mathrm{out}) + \mathcal{O}(\norm{\Laplace v_\mathrm{glue}}_\infty) + \mathcal{O}(\norm{\Laplace w_\mathrm{out}}_\infty\norm{ \nabla v_\mathrm{glue}}_\infty^2) \\
&+\mathcal{O}(\norm{D^2 w_\mathrm{out}}_\infty \norm{\nabla w_\mathrm{out}}_\infty^2 \norm{\nabla v_\mathrm{glue}}_\infty^2) \\
&+\mathcal{O}(\norm{D^2 v_\mathrm{glue}}_\infty ( \norm{\nabla w_\mathrm{out}}_\infty^2 +\norm{\nabla v_\mathrm{glue}}_\infty^2)   \\
&+ \mathcal{O}(\norm{D^2 w_\mathrm{out}}_\infty ( \norm{\nabla v_\mathrm{glue}}_\infty\norm{\nabla w_\mathrm{out}}_\infty +\norm{\nabla v_\mathrm{glue}}_\infty^2 ).
\end{align*}
Here, $(\cdot,\cdot)$ denotes the scalar product in $\R^n$ and $D^2 v \cdot \nabla w$ denotes the matrix-vector product of the matrix of second derivatives of $v\colon \R^n\to \R$ applied to the gradient vector of $w\colon \R^n\to\R$.
Note that $\norm{\nabla  v_\mathrm{glue}}_\infty = \mathcal{O}(\frac{h}{d})$, $\norm{D^2 v_\mathrm{glue}}_\infty = \mathcal{O}(\frac{h}{d^2})$ (Proposition~\ref{prop:glue}), and
$\norm{\nabla  u_\mathrm{out}}_\infty= \mathcal{O}(c)$. One can see that also $\norm{D^2 u_\mathrm{out}}_\infty = \mathcal{O}(c)$, since for small curvature the gradient term dominates when calculating the second derivatives of the function $v_\mathrm{out}$. All error terms can be made small with respect to $\abs{\kappa(w_\mathrm{out})} = \abs{F_\mathrm{out}} = c \frac{r_\mathrm{in}^{n-1}}{\left(\sqrt{n}(C_0 h^{-1/n}+\frac{d}{2}-r_1)\right)^n}$, by
noting that one can, instead of~\eqref{eq:MCF_scaling}, for a given $C>0$, fix $h$ and $d$ so that $-F_\mathrm{out} > C\frac{h}{d}$.
The rest of the proof then follows that of Theorem~\ref{thm:pinning}. 
\end{proof}

\section{Conclusions}
\label{sec:conclusion}
We have shown that, for our models of interface evolution in random media, a finite critical force is required to propagate the interface through the body. Many questions in this area, however, remain open. It was shown in~\cite{Coville_10a}, that for a model with obstacles on lattice sites\footnote {As pointed out in Remark~\ref{rem:lattice}, our construction of a supersolution for sufficiently small external driving force also works in this `lattice case' for $n=1$.} with random exponentially distributed strength for $n=1$, no more stationary solution can exist if the forcing exceeds a critical value. The question whether interfaces in this case move with a finite speed of propagation is still open and currently under investigation (this is of course trivial for uniformly bounded obstacle strength when also avoiding overlap of obstacles). These two results together would show that there is a transition from a viscous kinetic relation (i.e., {\it velocity{\rm\,=\,}force}) in the microscopic model turns (after a time-rescaling) into a rate independent model for the macroscopic behavior of the system. Such rate independent kinetics are commonly assumed in macroscopic models of phase transformations or plasticity. This article provides a step into deriving this assumption from microscopic viscous kinetics.

\section*{Acknowledgements}
P.~Dondl and M.~Scheutzow acknowledge support from the DFG-funded research group `Analysis and Stochastics in Complex Physical Systems' (FOR 718).


\begin{thebibliography}{DDG{\etalchar{+}}10}

\bibitem[BN04]{Brazovskii_04a}
S.~Brazovskii and T.~Nattermann, \emph{{Pinning and sliding of driven elastic
  systems: from domain walls to charge density waves}}, Advances in Physics
  \textbf{{53}} ({2004}), no.~{2}, {177--252}

\bibitem[CB04]{Craciun_04a}
B.~Craciun and K.~Bhattacharya, \emph{Effective motion of a curvature-sensitive
  interface through a heterogeneous medium}, Interfaces {F}ree {B}ound.
  \textbf{6} (2004), no.~2, 151--173. \MR{MR2079601 (2005e:35110)}

\bibitem[CDL10]{Coville_10a}
J.~Coville, N.~Dirr, and S.~Luckhaus, \emph{Non-existence of
  positive stationary solutions for a class of semi-linear {P}{D}{E}s with
  random coefficients}, Networks and Heterogeneous Media (to appear
  (Dec~2010)).

\bibitem[CLS09]{Cardaliaguet_09a}
P.~Cardaliaguet, P.-L. Lions, and P.~E. Souganidis, \emph{A discussion about
  the homogenization of moving interfaces}, J. Math. Pures Appl. (9)
  \textbf{91} (2009), no.~4, 339--363. \MR{MR2518002}

\bibitem[DDG{\etalchar{+}}10]{Dondl_09f}
N.~Dirr, P.~W. Dondl, G.~R. Grimmett, A.~E. Holroyd, and M.~Scheutzow,
  \emph{Lipschitz percolation}, Electronic Communications in Probability
  \textbf{15} (2010), 14--21.

\bibitem[Del41]{Delaunay_41a}
C.~Delaunay, \emph{Sur la surface de r{\'e}volution dont la courbure moyenne
  est constante}, J. {M}ath. {P}ures {A}ppl. \textbf{1} (1841), no.~6,
  309--320.

\bibitem[DKY08]{Dirr_08a}
N.~Dirr, G.~Karali, and N.~K. Yip, \emph{Pulsating wave for mean curvature flow
  in inhomogeneous medium}, European J. Appl. Math. \textbf{19} (2008), no.~6,
  661--699. \MR{2463225 (2009m:35224)}

\bibitem[DY06]{Dirr_06a}
N.~Dirr and N.~K. Yip, \emph{Pinning and de-pinning phenomena in front
  propagation in heterogeneous media}, Interfaces Free Bound. \textbf{8}
  (2006), no.~1, 79--109. \MR{MR2231253 (2007d:35144)}

\bibitem[GH10]{Grimmett_10b}
G.~R.~Grimmett and A.~E.~Holroyd, \emph{Geometry of lipschitz
  percolation}, arXiv:1007.3762v1 [math.PR] (2010).

\bibitem[Kar97]{Kardar_97a}
M.~Kardar, \emph{Nonequilibrium dynamics of interfaces and lines},
  arXiv:cond-mat/9704172v1 [cond-mat.stat-mech] (1997).

\bibitem[Kle04]{Kleemann_04a}
W.~Kleemann, \emph{Dynamic phase transitions in ferroic systems with pinned
  domain wall}, Oberwolfach Reports, vol.~1, 2004, pp.~1587--1656.

\end{thebibliography}

\newcommand{\etalchar}[1]{$^{#1}$}
\def\cprime{$'$}
\providecommand{\bysame}{\leavevmode\hbox to3em{\hrulefill}\thinspace}
\providecommand{\MR}{\relax\ifhmode\unskip\space\fi MR }
\providecommand{\MRhref}[2]{%
  \href{http://www.ams.org/mathscinet-getitem?mr=#1}{#2}
}
\providecommand{\href}[2]{#2}

\end{document}